\documentclass[review]{elsarticle}
\usepackage{stmaryrd}
\usepackage{bbm}
\usepackage[dvipsnames]{xcolor}
\usepackage{amsmath}
\usepackage{amsthm}
\usepackage{amssymb}
\usepackage{mathtools}
\usepackage{subcaption}
\usepackage{etoolbox}
\usepackage{environ}
\usepackage{ulem}
\usepackage{multicol}
\usepackage{listings}
\usepackage{graphicx}
\usepackage{multirow}
\usepackage{subcaption}
\usepackage{blkarray}
\usepackage{algorithm}
\usepackage{algpseudocode}

\usepackage{lineno,hyperref}
\modulolinenumbers[5]

\newcommand{\bc}{\operatorname{bc}}

\newcommand{\bp}{\operatorname{bp}}
\newcommand{\mc}{\operatorname{mc}}

\newcommand{\lab}{\operatorname{label}}

\newcommand{\MID}{\operatorname{mid}}

\newtheorem*{rep@theorem}{\rep@title}
\newcommand{\newreptheorem}[2]{%
\newenvironment{rep#1}[1]{%
 \def\rep@title{#2 \ref{##1}}%
 \begin{rep@theorem}}%
 {\end{rep@theorem}}}
\makeatother

\newtheorem{corollary}{Corollary}
\newtheorem{theorem}{Theorem}
\newreptheorem{theorem}{Theorem}
\newtheorem{lemma}{Lemma}
\newreptheorem{lemma}{Lemma}
\newtheorem{remark}{Remark}
\newtheorem{definition}{Definition}

\newtheorem{claim}{Claim}
\newreptheorem{claim}{Claim}
\newtheorem{proposition}{Proposition}
\newreptheorem{proposition}{Proposition}

\usepackage{tikz,tkz-graph}
\tikzstyle{vertex}=[circle, draw, inner sep=0pt, minimum size=1.5em]
\tikzstyle{svertex}=[draw, inner sep=0pt, minimum size=1.5em]

\usetikzlibrary{decorations}
\usetikzlibrary{arrows.meta,patterns}

\journal{ }

\begin{document}

\begin{frontmatter}

\title{Finding Biclique Partitions of Co-Chordal Graphs}

\author[1]{Bochuan Lyu\corref{cor1}}
\ead{bl46@rice.edu}
\author[1]{Illya V. Hicks}
\ead{ivhicks@rice.edu}
\cortext[cor1]{Corresponding author}
\address[1]{Rice University, Department of Computational Applied Mathematics and Operations Research, United States of America}

\begin{abstract}
The biclique partition number $(\bp)$ of a graph $G$ is referred to as the least number of complete bipartite (biclique) subgraphs that are required to cover the edges of the graph exactly once. In this paper, we show that the biclique partition number ($\bp$) of a co-chordal (complementary graph of chordal) graph $G = (V, E)$ is less than the number of maximal cliques ($\mc$) of its complementary graph: a chordal graph $G^c = (V, E^c)$. We first provide a general framework of the ``divide and conquer" heuristic of finding minimum biclique partitions of co-chordal graphs based on clique trees. Furthermore, a heuristic of complexity $O[|V|(|V|+|E^c|)]$ is proposed by applying lexicographic breadth-first search to find structures called moplexes. Either heuristic gives us a biclique partition of $G$ with size $\mc(G^c)-1$. In addition, we prove that both of our heuristics can solve the minimum biclique partition problem on $G$ exactly if its complement $G^c$ is chordal and clique vertex irreducible. We also show that $\mc(G^c) - 2 \leq \bp(G) \leq \mc(G^c) - 1$ if $G$ is a split graph.
\end{abstract}


\end{frontmatter}

\nolinenumbers

\section{Introduction}

The biclique partition number $(\bp)$ of a graph $G$ is referred to as the least number of complete bipartite (biclique) subgraphs that are required to cover the edges of the graph exactly once (however, the vertices can belong to two or more bicliques). The set of such biclique subgraphs is called a biclique partition of $G$. Graham and Pollak first introduced this concept in network addressing~\cite{graham1971addressing} and graph storage~\cite{graham1972embedding}. Their famous Graham-Pollak Theorem proves a result about the biclique partition numbers on complete graphs and it draws much attention from algebraic graph theory~\cite{cioabua2013variations, leader2017improved, peck1984new, tverberg1982decomposition, vishwanathan2010counting}. However, no purely combinatorial proof is known to the result~\cite{martin2018proofs}. Rawshdeh and Al-Ezeh~\cite{rawshdehbiclique} extended Graham-Pollak Theorem to find biclique partition numbers on line graphs and their complements of complete graphs and bicliques. The biclique partition also has a strong connection with biclique cover number $(\bc)$, where the edges of a graph are covered by bicliques but not necessarily disjointed. Pinto~\cite{pinto2013biclique} showed that $\bp(G) \leq \frac{1}{2}(3^{\bc(G)} - 1)$.

Moreover, finding a biclique partition with the minimum size is NP-complete even on the graphs without 4-cycles~\cite{kratzke1988eigensharp}. In this work, we focus on studying the biclique partition number on co-chordal (complement of chordal) and its important subclass, split graphs\footnote{Almost all chordal graphs are split graphs. It means that as $n$ goes to infinity, the fraction of $n$-vertex split graphs in $n$-vertex chordal graphs goes to 1~\cite{bender1985almost}. Since the complement of a split graph is also split, a split graph is also co-chordal.}.


There are also many related research studies around biclique partitions. Motivated by a technique for clustering data on binary matrices, Bein et al.~\cite{bein2008clustering} considered a biclique vertex partition problem on a bipartite graph where each vertex is covered exactly once in a collection of biclique subgraphs. De Sousa Filho et al.~\cite{de2021biclique} also studied the biclique vertex partition problem and its variant bicluster editing problem, where they developed a polyhedral study on biclique vertex partitions on a complete bipartite graph. Groshaus et al.~\cite{groshaus2022biclique} gave a polynomial-time algorithm to determine whether a graph is a biclique graph (an intersection graph of the bicliques) of a subclass of split graphs. Shigeta and Amano~\cite{shigeta2015ordered} provided an explicit construction of an ordered biclique partition, a variant of biclique partition, of $K_n$ of size $n^{1/2+o(1)}$, which improved the $O(n^{2/3})$ bound shown by Amano~\cite{amano2014some}.

Another related graph characteristic is $\bp_k(G)$ where the edges can be covered by at least one and at most $k$ biclique subgraphs~\cite{alon1997neighborly} and Alon showed that the minimum possible number for $\bp_k(K_n)$ is $\Theta(k n^{1/k})$, where $K_n$ is a complete graph with $n$ vertices. Recent work by Rohatgi et al.~\cite{rohatgi2020regarding} showed that if each edge is exactly covered by $k$ bicliques, the number of bicliques required to cover $K_n$ is $ (1 + o(1))n$.


In this paper, we study a biclique (edge) partition problem on co-chordal graphs and split graphs using clique trees, moplexes, and lexicographic breadth-first search (LexBFS) defined in Section~\ref{sec:pre}. We introduce a new definition, partitioning biclique, which can naturally naturally partition a graph into two induced subgraphs with no shared edges in Section~\ref{sec:PB}. In Section~\ref{sec:h_ct}, we provide a heuristic to find a biclique partition on a co-chordal graph given a clique tree of the complement of the co-chordal graph. We also prove the correctness of the heuristic and show the size of the biclique partition is exactly equal to the number of maximal cliques in the complementary graph of the co-chordal graph minus one. We also provide a corollary that states that the biclique partition number of a co-chordal graph is less than the number of maximal cliques of its complement. In Section~\ref{sec:h_peo}, we provide an efficient heuristic to obtain biclique partitions of co-chordal graphs by finding moplexes~\cite{berry1998separability}, defined later, using LexBFS. We also show that two heuristics provide biclique partitions of the same size.  In Section~\ref{sec:lower}, we prove that both heuristics can find a minimum biclique partition of $G$ if its complement $G^c$ is chordal and clique vertex irreducible. We also derive a lower bound of the biclique partition number of split graphs and show that our heuristics can obtain a biclique partition on any split graph with a size no more than the biclique partition number plus one. In Section~\ref{sec:fr}, we summarize the contribution of our work and point out some future directions.



\section{Preliminaries} \label{sec:pre}

A \textit{simple graph} is a pair $G = (V, E)$ where $V$ is a finite set of vertices and $E \subseteq \{uv: u, v\in V, u \neq v\}$. We use $V(G)$ and $E(G)$ to represent the vertex set and edge set of the graph $G$. Two vertices are \textit{adjacent} in $G$ if there is an edge between them. A \textit{subgraph} $G' = (V', E')$ of $G$ is a graph where $V' \subseteq V$ and $E' \subseteq \{uv \in E: u, v \in V'\}$. Given $A \subseteq V$, the \textit{subgraph} of $G$ \textit{induced} by $A$ is denoted as $G(A) = (A, E_A)$, where $E_A = \{uv \in E: u, v \in A\}$. A \textit{clique} is a subset of vertices of an graph $G$ such that every two distinct vertices are adjacent in $G$. The \textit{neighborhood} of a vertex $v$ of a graph $G$ is the set of all vertices, other than $v$, that are adjacent with $v$ and is denoted as $N_G(v)$. The \textit{closed neighborhood} of $v$ is denoted as $N_G[v] = N_G(v) \cup \{v\}$. For simplicity, we also define the neighborhood of a vertex set $A$ of a graph $G$ as $N_G(A) = \{u \in N_G(v): \forall v \in A\} \setminus A$ and closed neighborhood of $A$ as $N_G[A] = N_G(A) \cup A$. A \textit{maximal clique} of $G$ is a clique of $G$ such that it is not a proper subset of any clique of $G$. Note that a vertex set with only one vertex $K_1$ is also a clique. We denote the number of maximal cliques of $G$ as $\mc(G)$ and we use $\mathcal{K}_G$ or $\mathcal{K}$ to denote the set of all maximal cliques of $G$. A maximum clique of $G$ is a clique of $G$ with the maximum number of vertices and we denote that number as the clique number of $G$, $\omega(G)$. An \textit{independent set} of a graph $G$ is a set of vertices that are pairwise nonadjacent with each other in $G$. Similarly, a \textit{maximal independent set} is an independent set that is not a proper subset of any independent set and a \textit{maximum independent set} is an independent set with the maximum number of vertices. Note that an independent set can be empty or only have one vertex.

A graph $C_n = (V, E)$ is a \textit{cycle} if the vertices and edges: $V = \{v_1, v_2, \hdots, v_n\}$ and $E = \{v_1v_2, v_2v_3, \hdots, v_{n-1}v_n, v_nv_1\}$. A graph is a \textit{tree} if it is connected and does not have any subgraph that is a cycle.

Given two vertex sets $U$ and $V$, we denote $U \times V$ to be the edge set $\{uv: u \in U, v \in V\}$. A \textit{bipartite} graph $G = (L \cup R, E)$ is a graph where $L$ and $R$ are disjointed vertex sets with the edge set $E \subseteq L \times R$. A \textit{biclique} graph is a complete bipartite graph $G = (L \cup R, E)$ where $E = L \times R$ and we denote it as $\{L, R\}$ for short. A \textit{biclique partition} of a graph $G$ is a collection of biclique subgraphs of $G$ such that every edge of $G$ is in exactly one biclique of the collection. The \textit{minimum biclique partition problem} on $G$ is to find a biclique partition with the minimum number of bicliques in the collection and we denote that value to be $\bp(G)$.

We denote that $\llbracket n \rrbracket = \{1, 2, \hdots, n\}$ where $n$ is a positive integer. A vertex is \textit{simplicial} if its neighborhood is a clique. An ordering $v_1, v_2, \hdots, v_n$ of $V$ is a \textit{perfect elimination ordering} if for all $i \in \llbracket n \rrbracket$, $v_i$ is simplicial on the induced subgraph $G(\{v_j: j \in \{i, i+1, \hdots, n\}\})$. An ordering function $\sigma: \llbracket n \rrbracket \rightarrow V$ is defined to describe the ordering of vertices $V$. A \textit{clique tree} $\mathcal{T}_{\mathcal{K}}$ for a chordal graph $G$ is a tree where each vertex represents a maximal clique of $G$ and satisfies the clique-intersection property: given any two distinct maximal cliques $K_1$ and $K_2$ in the tree, every clique on the path between $K_1$ and $K_2$ in $\mathcal{T}_{\mathcal{K}}$ contains $K_1 \cap K_2$. We also define the \textit{middle set} of edge $e \in \mathcal{T}_{\mathcal{K}}$, $\MID(e)$, to be the intersection of the vertices of cliques on its two ends.

A graph $G$ is \textit{clique vertex irreducible} if every maximal clique in $G$ has a vertex which does not lie in any other maximal clique of $G$~\cite{lakshmanan2009clique}.




A \textit{split} graph is a graph whose vertices can be partitioned into a clique and an independent set. A graph is \textit{chordal} if there is no induced cycle subgraph of length greater than 3. Note that a split graph is also chordal. A graph is chordal if and only if it has a clique tree~\cite{blair1993introduction}. Also, a graph is chordal if and only if it has a perfect elimination ordering~\cite{fulkerson1965incidence}. A perfect elimination ordering of a chordal graph can be obtained by \textit{lexicographic breadth-first search} (LexBFS) described in Algorithm~\ref{alg:lexbfs}~\cite{corneil2004lexicographic, rose1976algorithmic}, where lexicographical order is defined in Definition~\ref{def:lex_order}. Note that the LexBFS algorithm in Algorithm~\ref{alg:lexbfs} can be implemented in linear-time: $O(|V| + |E|)$ with partition refinement~\cite{habib2000lex}.

\begin{definition} \label{def:lex_order} [Lexicographical Order]
Let $X$ be a set of all vectors of real numbers with a finite length. Then, we can define a lexicographical order on $X$ where
\begin{enumerate}
    \item $\emptyset \in X$ and for any $x \in X$, $\emptyset \preccurlyeq x$.
    \item $(x_1, x_2) \preccurlyeq (y_1, y_2)$ if and only if one of the following holds:
    \begin{enumerate}
        \item $x_1 < y_1$
        \item $x_1 = y_1$ and $x_2 \preccurlyeq y_2$
    \end{enumerate}
    \noindent where $x_1, y_1 \in \mathbb{R}$ and $x_2, y_2 \in X$. 
\end{enumerate}
\end{definition}

\begin{algorithm}[H]
\begin{algorithmic}[1]
\State \textbf{Input}: A graph $G = (V, E)$ and an arbitrary selected vertex $v$ in $V$.
\State \textbf{Output}: An ordering function $\sigma: \llbracket n \rrbracket \rightarrow V$ of the vertices $V$.
\State $\lab(v) \leftarrow (n)$ where $n = |V|$, and $\lab(u) \leftarrow ()$ for all $u \in V \setminus \{v\}$.
\For{$i \in \{n, n-1, \hdots, 1\}$}
\State Select an unnumbered vertex $u$ with lexicographically the largest label.
\State $\sigma(i) \leftarrow u$.
\For{each unnumbered vertex $w$ in $N_G(u)$}
\State $\lab(w) \leftarrow (\lab(w), i)$. \Comment{$(\cdot, \cdot)$ here is a concatenation operation.}
\EndFor
\EndFor
\State \textbf{return} $\sigma$
\end{algorithmic}
\caption{Generic Lexicographic Breadth-First Search~\cite{corneil2004lexicographic, rose1976algorithmic}.} \label{alg:lexbfs}
\end{algorithm}

In our work, we study biclique partitions on \textit{co-chordal} graphs, so most of the $G^c$'s in the following sections are chordal graphs. We use $\mathcal{K}^c$ and $\mathcal{T}_{\mathcal{K}^c}$ to denote the set of all maximal cliques and a clique tree of $G^c$.

A \textit{module} of a graph $G = (V, E)$ is a vertex set $A \subseteq V$ such that all the vertices in $A$ share the same neighborhood in $V \setminus A$. A \textit{separator} of $G = (V, E)$ is a set of vertices, say $S$, such that $G(V \setminus S)$ is disconnected. A separator $S$ is minimal if no proper set of $S$ is a separator. A \textit{moplex} $X$ of $G$ is both a clique and a module such that $N_G(X)$ is a minimal separator (see Figure~\ref{fig:moplex}). Berry and Bordat~\cite{berry1998separability} discovered that LexBFS can be applied to find a moplex of a general graph $G$ in linear time, which motivates us to propose the biclique partition algorithm in Section~\ref{sec:h_peo}.

\begin{proposition} [Theorem 5.1~\cite{berry1998separability}] \label{prop:sigma_1_moplex}
Given a graph $G$ and $\sigma$ generated by Algorithm~\ref{alg:lexbfs}, $\sigma(1)$ belongs to a moplex.
\end{proposition}

\begin{figure}[H]
    \centering
\tikzset{every picture/.style={line width=0.75pt}} 

\begin{tikzpicture}[x=0.75pt,y=0.75pt,yscale=-1,xscale=1]
\draw  [dash pattern={on 4.5pt off 4.5pt}] (73.55,108.51) .. controls (73.65,98.43) and (98.27,90.52) .. (128.52,90.84) .. controls (158.78,91.15) and (183.23,99.58) .. (183.12,109.66) .. controls (183.01,119.73) and (158.4,127.65) .. (128.14,127.33) .. controls (97.88,127.01) and (73.44,118.59) .. (73.55,108.51) -- cycle ;
\draw  [fill={rgb, 255:red, 0; green, 0; blue, 0 }  ,fill opacity=1 ] (93.33,109.08) .. controls (93.33,106.32) and (95.57,104.08) .. (98.33,104.08) .. controls (101.09,104.08) and (103.33,106.32) .. (103.33,109.08) .. controls (103.33,111.84) and (101.09,114.08) .. (98.33,114.08) .. controls (95.57,114.08) and (93.33,111.84) .. (93.33,109.08) -- cycle ;
\draw  [fill={rgb, 255:red, 0; green, 0; blue, 0 }  ,fill opacity=1 ] (153.33,109.08) .. controls (153.33,106.32) and (155.57,104.08) .. (158.33,104.08) .. controls (161.09,104.08) and (163.33,106.32) .. (163.33,109.08) .. controls (163.33,111.84) and (161.09,114.08) .. (158.33,114.08) .. controls (155.57,114.08) and (153.33,111.84) .. (153.33,109.08) -- cycle ;
\draw  [fill={rgb, 255:red, 0; green, 0; blue, 0 }  ,fill opacity=1 ] (93.33,169.08) .. controls (93.33,166.32) and (95.57,164.08) .. (98.33,164.08) .. controls (101.09,164.08) and (103.33,166.32) .. (103.33,169.08) .. controls (103.33,171.84) and (101.09,174.08) .. (98.33,174.08) .. controls (95.57,174.08) and (93.33,171.84) .. (93.33,169.08) -- cycle ;
\draw  [fill={rgb, 255:red, 0; green, 0; blue, 0 }  ,fill opacity=1 ] (29,184.42) .. controls (29,181.66) and (31.24,179.42) .. (34,179.42) .. controls (36.76,179.42) and (39,181.66) .. (39,184.42) .. controls (39,187.18) and (36.76,189.42) .. (34,189.42) .. controls (31.24,189.42) and (29,187.18) .. (29,184.42) -- cycle ;
\draw  [fill={rgb, 255:red, 0; green, 0; blue, 0 }  ,fill opacity=1 ] (72.33,229.08) .. controls (72.33,226.32) and (74.57,224.08) .. (77.33,224.08) .. controls (80.09,224.08) and (82.33,226.32) .. (82.33,229.08) .. controls (82.33,231.84) and (80.09,234.08) .. (77.33,234.08) .. controls (74.57,234.08) and (72.33,231.84) .. (72.33,229.08) -- cycle ;
\draw  [fill={rgb, 255:red, 0; green, 0; blue, 0 }  ,fill opacity=1 ] (213.33,229.08) .. controls (213.33,226.32) and (215.57,224.08) .. (218.33,224.08) .. controls (221.09,224.08) and (223.33,226.32) .. (223.33,229.08) .. controls (223.33,231.84) and (221.09,234.08) .. (218.33,234.08) .. controls (215.57,234.08) and (213.33,231.84) .. (213.33,229.08) -- cycle ;
\draw  [fill={rgb, 255:red, 0; green, 0; blue, 0 }  ,fill opacity=1 ] (213.33,169.08) .. controls (213.33,166.32) and (215.57,164.08) .. (218.33,164.08) .. controls (221.09,164.08) and (223.33,166.32) .. (223.33,169.08) .. controls (223.33,171.84) and (221.09,174.08) .. (218.33,174.08) .. controls (215.57,174.08) and (213.33,171.84) .. (213.33,169.08) -- cycle ;
\draw  [fill={rgb, 255:red, 0; green, 0; blue, 0 }  ,fill opacity=1 ] (153.33,229.08) .. controls (153.33,226.32) and (155.57,224.08) .. (158.33,224.08) .. controls (161.09,224.08) and (163.33,226.32) .. (163.33,229.08) .. controls (163.33,231.84) and (161.09,234.08) .. (158.33,234.08) .. controls (155.57,234.08) and (153.33,231.84) .. (153.33,229.08) -- cycle ;
\draw  [fill={rgb, 255:red, 0; green, 0; blue, 0 }  ,fill opacity=1 ] (153.33,169.08) .. controls (153.33,166.32) and (155.57,164.08) .. (158.33,164.08) .. controls (161.09,164.08) and (163.33,166.32) .. (163.33,169.08) .. controls (163.33,171.84) and (161.09,174.08) .. (158.33,174.08) .. controls (155.57,174.08) and (153.33,171.84) .. (153.33,169.08) -- cycle ;
\draw    (98.33,109.08) -- (158.33,169.08) ;
\draw    (98.33,169.08) -- (158.33,169.08) ;
\draw    (98.33,169.08) -- (77.33,229.08) ;
\draw    (34,184.42) -- (98.33,169.08) ;
\draw    (34,184.42) -- (77.33,229.08) ;
\draw    (98.33,109.08) -- (98.33,169.08) ;
\draw    (98.33,109.08) -- (158.33,109.08) ;
\draw    (158.33,109.08) -- (158.33,169.08) ;
\draw    (158.33,109.08) -- (98.33,169.08) ;
\draw    (218.33,169.08) -- (158.33,169.08) ;
\draw    (158.33,169.08) -- (158.33,229.08) ;
\draw    (218.33,169.08) -- (218.33,229.08) ;
\draw    (158.33,229.08) -- (218.33,229.08) ;
\draw    (158.33,169.08) -- (218.33,229.08) ;
\draw    (218.33,169.08) -- (158.33,229.08) ;
\draw  [dash pattern={on 4.5pt off 4.5pt}] (217.93,144.3) .. controls (228.01,144.22) and (236.36,168.69) .. (236.58,198.95) .. controls (236.8,229.21) and (228.81,253.8) .. (218.73,253.87) .. controls (208.66,253.94) and (200.31,229.48) .. (200.09,199.22) .. controls (199.86,168.96) and (207.85,144.37) .. (217.93,144.3) -- cycle ;
\draw    (98.33,169.08) -- (158.33,229.08) ;
\draw (137.37,73.5) node   [align=left] {\begin{minipage}[lt]{21.35pt}\setlength\topsep{0pt}
$\displaystyle X_{1}$
\end{minipage}};
\draw (258.7,192.83) node   [align=left] {\begin{minipage}[lt]{21.35pt}\setlength\topsep{0pt}
$\displaystyle X_{2}$
\end{minipage}};
\end{tikzpicture}
    \caption{$X_1$ is a module but not a moplex. $X_2$ is a moplex.}
    \label{fig:moplex}
\end{figure}
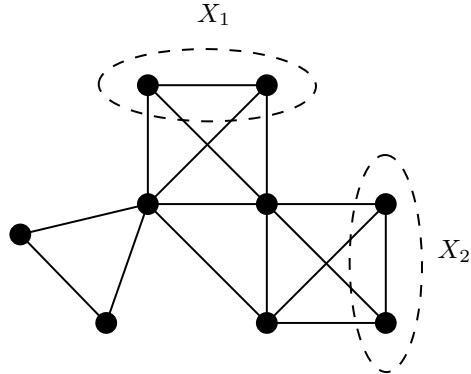



\section{Partitioning Biclique} \label{sec:PB}

We will introduce a new definition: partitioning biclique, which can naturally partition the edges of a graph into two edge disjoint subgraphs. We first start with a definition of partitioning biclique.

\begin{definition}[partitioning biclique]
Given a graph $G = (V, E)$, a biclique subgraph of $G$, $\{L, R\}$, is a partitioning biclique subgraph of $G$ if the edge sets of $\{L, R\}$, $G(V \setminus L)$, and $G(V \setminus R)$ partition the edges of $G$, i.e. each edge of $G$ is exactly in one of $\{L, R\}$, $G(V \setminus L)$, and $G(V \setminus R)$.
\end{definition}

It is worth to noticing that $\{L, R\}$ is not necessarily equal to the induced subgraph $G(L \cup R)$. It means that $L$ and $R$ might not be independent sets of $G$. 

We then show that an arbitrary biclique subgraph can divide the edges of the original graphs into three parts: the biclique and two induced subgraphs. See Figure~\ref{fig:partition_biclique} for a visualization of this idea: the edges of graph $G$ are partitioned into the edges of $\{L, R\}$, $G(V \setminus L)$, and $G(V \setminus R)$.

\begin{lemma} \label{lm:biclique_sep}
Given a graph $G = (V, E)$, let $\{L, R\}$ be an arbitrary biclique subgraph of $G$. Then, $E = E(\{L, R\}) \cup E(G(V \setminus L)) \cup E(G(V \setminus R))$.
\end{lemma}

\begin{proof}
Since $\{L, R\}$, $G(V \setminus L)$, $G(V \setminus R)$ are all subgraphs of $G$, then $E(\{L, R\}) \cup E(G(V \setminus L)) \cup E(G(V \setminus R)) \subseteq E$.

Let $C = V \setminus (L \cup R)$ and $uv$ be an arbitrary edge in $G$. Thus, $V \setminus L = R \cup C$ and $V \setminus R = L \cup C$. If $u, v \in (L \cup C)$, then $uv \in E(G(V \setminus R))$. If neither of $u$ and $v$ is in $(L \cup C)$, then $u, v \in R$ and $uv \in E(G(V \setminus L))$. If one of $u$ and $v$ is in $L \cup C$ and the other one is not, then we can assume that $u \in (L \cup C)$ and $v \in R$ without loss of generality. Then, either $uv \in E(G(V \setminus L))$ ($u \in C$) or $uv \in \{L, R\}$ ($u \in L$). Hence, $E(\{L, R\}) \cup E(G(V \setminus L)) \cup E(G(V \setminus R)) = E$.
\end{proof}

Next, we show that a biclique subgraph is a partitioning biclique of a graph $G$ if and only if the vertices in $G$ that are not in the biclique form an independent set in $G$.

\begin{figure}[H]
    \centering
\tikzset{every picture/.style={line width=0.75pt}} 

\begin{tikzpicture}[x=0.75pt,y=0.75pt,yscale=-1,xscale=1]

\draw  [fill={rgb, 255:red, 0; green, 0; blue, 0 }  ,fill opacity=1 ] (69,91) .. controls (69,88.24) and (71.24,86) .. (74,86) .. controls (76.76,86) and (79,88.24) .. (79,91) .. controls (79,93.76) and (76.76,96) .. (74,96) .. controls (71.24,96) and (69,93.76) .. (69,91) -- cycle ;
\draw  [fill={rgb, 255:red, 0; green, 0; blue, 0 }  ,fill opacity=1 ] (69,131) .. controls (69,128.24) and (71.24,126) .. (74,126) .. controls (76.76,126) and (79,128.24) .. (79,131) .. controls (79,133.76) and (76.76,136) .. (74,136) .. controls (71.24,136) and (69,133.76) .. (69,131) -- cycle ;
\draw  [fill={rgb, 255:red, 0; green, 0; blue, 0 }  ,fill opacity=1 ] (69,171) .. controls (69,168.24) and (71.24,166) .. (74,166) .. controls (76.76,166) and (79,168.24) .. (79,171) .. controls (79,173.76) and (76.76,176) .. (74,176) .. controls (71.24,176) and (69,173.76) .. (69,171) -- cycle ;
\draw  [fill={rgb, 255:red, 0; green, 0; blue, 0 }  ,fill opacity=1 ] (69,211) .. controls (69,208.24) and (71.24,206) .. (74,206) .. controls (76.76,206) and (79,208.24) .. (79,211) .. controls (79,213.76) and (76.76,216) .. (74,216) .. controls (71.24,216) and (69,213.76) .. (69,211) -- cycle ;
\draw  [fill={rgb, 255:red, 0; green, 0; blue, 0 }  ,fill opacity=1 ] (209,111) .. controls (209,108.24) and (211.24,106) .. (214,106) .. controls (216.76,106) and (219,108.24) .. (219,111) .. controls (219,113.76) and (216.76,116) .. (214,116) .. controls (211.24,116) and (209,113.76) .. (209,111) -- cycle ;
\draw  [fill={rgb, 255:red, 0; green, 0; blue, 0 }  ,fill opacity=1 ] (209,191) .. controls (209,188.24) and (211.24,186) .. (214,186) .. controls (216.76,186) and (219,188.24) .. (219,191) .. controls (219,193.76) and (216.76,196) .. (214,196) .. controls (211.24,196) and (209,193.76) .. (209,191) -- cycle ;
\draw    (79,91) -- (209,111) ;
\draw  [fill={rgb, 255:red, 0; green, 0; blue, 0 }  ,fill opacity=1 ] (186.37,266) .. controls (189.05,265.92) and (191.29,268.09) .. (191.38,270.85) .. controls (191.46,273.61) and (189.35,275.91) .. (186.67,276) .. controls (183.99,276.08) and (181.75,273.91) .. (181.67,271.15) .. controls (181.58,268.39) and (183.69,266.09) .. (186.37,266) -- cycle ;
\draw  [fill={rgb, 255:red, 0; green, 0; blue, 0 }  ,fill opacity=1 ] (147.54,267.22) .. controls (150.22,267.14) and (152.46,269.3) .. (152.54,272.06) .. controls (152.63,274.82) and (150.52,277.13) .. (147.84,277.21) .. controls (145.16,277.3) and (142.92,275.13) .. (142.83,272.37) .. controls (142.75,269.61) and (144.86,267.3) .. (147.54,267.22) -- cycle ;
\draw  [fill={rgb, 255:red, 0; green, 0; blue, 0 }  ,fill opacity=1 ] (108.71,268.44) .. controls (111.39,268.35) and (113.63,270.52) .. (113.71,273.28) .. controls (113.79,276.04) and (111.69,278.35) .. (109.01,278.43) .. controls (106.32,278.52) and (104.08,276.35) .. (104,273.59) .. controls (103.92,270.83) and (106.03,268.52) .. (108.71,268.44) -- cycle ;
\draw    (79,131) -- (209,111) ;
\draw    (79,171) -- (209,111) ;
\draw    (79,211) -- (209,111) ;
\draw    (79,91) -- (209,191) ;
\draw    (79,131) -- (209,191) ;
\draw    (79,171) -- (209,191) ;
\draw    (79,211) -- (209,191) ;
\draw    (69,91) .. controls (40.4,86.9) and (38.4,181.9) .. (69,171) ;
\draw    (69,131) .. controls (40.4,126.9) and (38.4,221.9) .. (69,211) ;
\draw    (209,191) -- (147.54,267.22) ;
\draw    (79,131) -- (147.54,267.22) ;
\draw    (209,111) -- (108.71,268.44) ;
\draw    (79,91) -- (186.37,266) ;
\draw    (209,111) -- (186.37,266) ;
\draw    (79,211) -- (108.71,268.44) ;
\draw  [dash pattern={on 4.5pt off 4.5pt}] (212.53,90.38) .. controls (223.57,89.96) and (233.59,117.96) .. (234.91,152.92) .. controls (236.24,187.88) and (228.36,216.56) .. (217.33,216.97) .. controls (206.29,217.39) and (196.27,189.39) .. (194.94,154.44) .. controls (193.62,119.48) and (201.49,90.8) .. (212.53,90.38) -- cycle ;
\draw  [dash pattern={on 4.5pt off 4.5pt}] (84.35,272.05) .. controls (84.38,261.01) and (112.76,252.13) .. (147.74,252.22) .. controls (182.72,252.31) and (211.06,261.34) .. (211.03,272.38) .. controls (211,283.43) and (182.62,292.31) .. (147.64,292.22) .. controls (112.65,292.13) and (84.32,283.1) .. (84.35,272.05) -- cycle ;
\draw  [dash pattern={on 4.5pt off 4.5pt}] (71.52,66.42) .. controls (82.56,66) and (93.02,105.48) .. (94.88,154.59) .. controls (96.74,203.71) and (89.3,243.86) .. (78.26,244.28) .. controls (67.22,244.7) and (56.77,205.22) .. (54.91,156.11) .. controls (53.04,106.99) and (60.48,66.84) .. (71.52,66.42) -- cycle ;

\draw (80.7,54.75) node   [align=left] {\begin{minipage}[lt]{21.35pt}\setlength\topsep{0pt}
$L$
\end{minipage}};
\draw (220.7,55.75) node   [align=left] {\begin{minipage}[lt]{21.35pt}\setlength\topsep{0pt}
$R$
\end{minipage}};
\draw (153.7,311.75) node   [align=left] {\begin{minipage}[lt]{21.35pt}\setlength\topsep{0pt}
$C$
\end{minipage}};

\end{tikzpicture}
    \caption{Given a graph $G = (V, E)$, $\{L, R\}$ is a partitioning biclique of $G$, where each edge in $E$ is exactly in one of $\{L, R\}$, $G(V \setminus L)$, and $G(V \setminus R)$.}
    \label{fig:partition_biclique}
\end{figure}
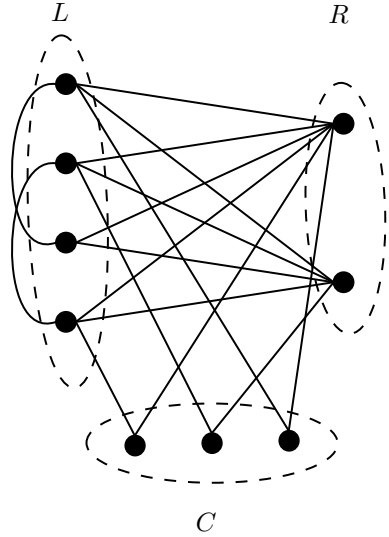

\begin{proposition} \label{prop:partitioned_biclique}
Given a graph $G = (V, E)$ and a biclique subgraph of $G$: $\{L, R\}$, then $\{L, R\}$ is a partitioning biclique subgraph of $G$ if and only if $V \setminus (L \cup R)$ is an independent set in $G$.
\end{proposition}

\begin{proof}
Denote $C= V \setminus (L \cup R)$. In the backward direction, suppose that $C$ is an independent set in $G$. By Lemma~\ref{lm:biclique_sep}, we know that $E = E(\{L, R\}) \cup E(G(V \setminus L)) \cup E(G(V \setminus R))$. Then, we need to show that $E(\{L, R\})$, $E(G(V \setminus L))$, $E(G(V \setminus R))$ are disjointed. Since every edge in $\{L, R\}$ is between vertices in $L$ and $R$, then $E(\{L, R\}) \cap E(G(V \setminus L)) = \emptyset = E(\{L, R\}) \cap E(G(V \setminus R))$. Since $C$ is an independent set of $G$ and 
\begin{align*}
    [(V \setminus L) \times (V \setminus L)] \cap [(V \setminus R) \times (V \setminus R)] = C \times C,
\end{align*}
\noindent we know that $E(G(V \setminus L)) \cap E(G(V \setminus R)) = \emptyset$.

In the forward direction, the proof is by contrapositive. Suppose that $C$ isn't an independent set in $G$. Then, there exists an edge $uv$ such that $u, v \in C$, where $uv$ is in both subgraphs $G(V \setminus L)$ and $G(V \setminus R)$. Hence, $C$ must be an independent set.
\end{proof}

Note that an independent set can be empty. Since the partitioning biclique can partition the original graph into a biclique subgraph and two induced subgraphs, it can be used to design a heuristic to find a biclique partition of a graph. In the next two sections, we will focus on a class of graph, co-chordal, where partitioning bicliques are easy to find since the complementary graph is chordal.

\section{A Heuristic Based on Clique Trees} \label{sec:h_ct}

In this section, we want to design a heuristic to find a biclique partition of a co-chordal graph $G$. Since $G^c$ is chordal, one of the good ways to represent $G^c$ is its clique tree where each vertex represents a maximal clique of $G$ and satisfies the clique-intersection property. We demonstrate a heuristic with an input of a clique tree $\mathcal{T}_{\mathcal{K}^c}$ of $G^c$ and an output of a biclique partition of $G$ in Algorithm~\ref{alg:bp_biclique_sep}. We also show that the size of that biclique partition is equal to $\mc(G^c)-1$, which provides us an upper bound of the biclique partition number of co-chordal graphs.

\begin{algorithm}[H]
\begin{algorithmic}[1]
\State \textbf{Input}: A clique tree $\mathcal{T}_{\mathcal{K}^c}$ of a chordal graph $G^c$.
\State \textbf{Output}: A biclique partition of the complementary graph $G$ of $G^c$.
\Function{FindPartition}{$\mathcal{T}_{\mathcal{K}^c}$}
\If{$|V(\mathcal{T}_{\mathcal{K}^c})| \leq 1$}
\State \textbf{return} $\emptyset$
\EndIf
\State Select an arbitrary edge $e$ to cut $\mathcal{T}_{\mathcal{K}^c}$ into two components $\mathcal{T}_{\mathcal{K}^c_1}$ and $\mathcal{T}_{\mathcal{K}^c_2}$. \label{ln:biclique_partition_sep_edge}
\State $L = \bigcup_{K \in V(\mathcal{T}_{\mathcal{K}^c_1})} K \setminus \MID(e)$; $R = \bigcup_{K \in V(\mathcal{T}_{\mathcal{K}^c_2})} K \setminus \MID(e)$. \label{ln:biclique_partition_sep_biclique}
\State \textbf{return} $\{\{L, R\}\} \cup \Call{FindPartition}{\mathcal{T}_{\mathcal{K}^c_1}} \cup \Call{FindPartition}{\mathcal{T}_{\mathcal{K}^c_2}}$
\EndFunction
\end{algorithmic}
\caption{Find a biclique partition of a co-chordal graph $G$ given a clique tree of $G^c$.} \label{alg:bp_biclique_sep}
\end{algorithm}

We prove that the output of Algorithm~\ref{alg:bp_biclique_sep} is a biclique partition of a co-chordal graph $G$ by showing that at each recursion a nonempty partitioning biclique $\{L, R\}$ is found and two subtrees $\mathcal{T}_{\mathcal{K}^c_1}$ and $\mathcal{T}_{\mathcal{K}^c_2}$ are also clique trees of two induced subgraphs $G^c(V \setminus L)$, and $G^c(V \setminus R)$ of $G^c$ respectively. Note that the edge $e$ can be selected arbitrarily in Algorithm~\ref{alg:bp_biclique_sep}.

\begin{proposition} \label{prop:bp_separation}
Given a chordal graph $G^c = (V, E^c)$ and one of its clique trees $\mathcal{T}_{\mathcal{K}^c} = (\mathcal{K}^c, \mathcal{E})$ where $V(\mathcal{T}_{\mathcal{K}^c}) > 1$, any edge $e$ of $\mathcal{T}_{\mathcal{K}^c}$ can partition $\mathcal{K}^c$ into $\mathcal{K}_1^c$ and $\mathcal{K}_2^c$ (trees $\mathcal{T}_{\mathcal{K}^c_1}$ and $\mathcal{T}_{\mathcal{K}^c_2}$ respectively) such that 
\begin{enumerate}
    \item[$\operatorname{(1)}$] The edges of $\{L, R\} = \{\bigcup_{K \in \mathcal{K}_1^c} K \setminus \MID(e), \bigcup_{K \in \mathcal{K}_2^c} K \setminus \MID(e) \}$, $G(V \setminus L)$, and $G(V \setminus R)$ partition the edges in $G$ where $G$ is the complementary graph of $G^c$, i.e. $\{L, R\}$ is a paritioned biclique subgraph of $G$.
    \item[$\operatorname{(2)}$] $\mathcal{T}_{\mathcal{K}^c_2}$ and $\mathcal{T}_{\mathcal{K}^c_1}$ are clique trees of chordal graphs $G^c(V \setminus L)$, and $G^c(V \setminus R)$ respectively.
    \item[$\operatorname{(3)}$] Both $L$ and $R$ are not empty.
\end{enumerate}
\end{proposition}

\begin{proof}
(1) is proved by Claim~\ref{cm:sep_1} and definition of partitioning biclique, (2) is proved by Claim~\ref{cm:sep_2b}, and (3) is proved by Claim~\ref{cm:sep_3}. 
\end{proof}



We next prove that $\{L, R\}$ is a partitioning biclique of $G$. Note that $L$ and $R$ do not have to be independent sets.

\begin{claim} \label{cm:sep_1}
$\{L, R\}$ is a partitioning biclique of the complementary graph $G = (V, E)$ of $G^c$.
\end{claim}

\begin{proof}
Since every edge in a tree is a cut, $e$ can partition $\mathcal{T}_{\mathcal{K}^c}$ into two sets of vertices, $\mathcal{K}_1^c$ and $\mathcal{K}_2^c$, in $\mathcal{T}_{\mathcal{K}^c}$. Let the two ends of edge $e$ in $\mathcal{T}_{\mathcal{K}^c}$ to be $K'$ and $K''$. Since $\MID(e) = K' \cap K''$ and both $K'$ and $K''$ are clique subgraphs of $G^c$, then $\MID(e)$ is an independent set of $G$.

Given an arbitrary $u \in L =\bigcup_{K \in \mathcal{K}_1^c} K \setminus \MID(e)$ and $v \in R = \bigcup_{K \in \mathcal{K}_2^c} K \setminus \MID(e)$, we assume that $uv \in E(G^c)$. Otherwise, $\{L, R\}$ is a biclique of $G$ and our result has been proved.

Since $uv \in E(G^c)$, we know that there exists some $K \in \mathcal{K}^c$ such that $\{u, v\} \subseteq K$. Without loss of generality, we assume that $u \in K_1$ and $\{u, v\} \subseteq K_2$ where $K_1 \in \mathcal{K}_1^c$ and $K_2 \in \mathcal{K}_2^c$. Since $K'$ and $K''$ both on the path between $K_1$ and $K_2$, $K_1 \cap K_2 \subseteq \MID(e)$. Therefore, $u \in \MID(e)$, which is a contradiction. Hence, $uv$ is not an edge of $G^c$ and $\{L, R\}$ is a biclique of $G$. 
\end{proof}

In Claim~\ref{cm:sep_2b}, we prove part (2) of Proposition~\ref{prop:bp_separation}. We first remark that any subgraph of a chordal graph is chordal and then use it to prove part (2) in Claim~\ref{cm:sep_2b}.

\begin{remark} \label{rm:sep_2a}
Given a chordal graph $G^c = (V, E^c)$, any induced subgraph of $G^c$ is chordal.
\end{remark}

\begin{claim} \label{cm:sep_2b}
$\mathcal{T}_{\mathcal{K}^c_2}$ and $\mathcal{T}_{\mathcal{K}^c_1}$ are clique trees of chordal graphs $G^c(V \setminus L)$, and $G^c(V \setminus R)$ respectively.
\end{claim}

\begin{proof}
By Remark~\ref{rm:sep_2a}, we know that both $G^c(V \setminus L)$, and $G^c(V \setminus R)$ are chordal. Thus, there exist clique trees for both $G^c(V \setminus L)$ and $G^c(V \setminus R)$. Without loss of generality, we only need to prove $\mathcal{T}_{\mathcal{K}^c_1}$ is a clique tree of $G^c(V \setminus R)$. Since $\mathcal{T}_{\mathcal{K}^c_1}$ is a subtree of $\mathcal{T}_{\mathcal{K}^c}$, we only need to show that $\mathcal{K}_1^c$ is the set of all maximal cliques of $G^c(V \setminus R)$. 

First, we want to prove that $\bigcup_{K \in \mathcal{K}_1^c} K = V \setminus R$. We proved in Claim~\ref{cm:sep_1}, $\{L, R\} = \{\bigcup_{K \in \mathcal{K}_1^c} K \setminus \MID(e), \bigcup_{K \in \mathcal{K}_2^c} K \setminus \MID(e) \}$ is a biclique. Thus, $L$ and $R$ are disjoint vertex sets. Since $V = \bigcup_{K \in \mathcal{K}^c} K$ and $L \cup R = \bigcup_{K \in \mathcal{K}^c} K \setminus \MID(e)$, then $\MID(e) = V \setminus (L \cup R)$. Thus,
\begin{align*}
\bigcup_{K \in \mathcal{K}_1^c} K = L \cup \MID(e)= V \setminus R.    
\end{align*}
Hence, $K_1$ is a maximal clique of $G^c(V \setminus R)$ for any $K_1 \in \mathcal{K}_1^c$. By the definition of clique tree, given an arbitrary $K_2 \in \mathcal{K}_2^c$ 
\begin{align*}
    K_2 \cap (V \setminus R) = K_2 \cap \left(\bigcup_{K \in \mathcal{K}_1^c} K\right) = \bigcup_{K \in \mathcal{K}_1^c} (K_2 \cap K) \subseteq \MID(e).
\end{align*}

Since $\MID(e) \subset K'$ for some $K' \in \mathcal{K}_1^c$, then $K_2 \cap (V \setminus R)$ cannot be a maximal clique of $G^c$. Therefore, $\mathcal{K}_1^c$ is the set of all maximal cliques of $G^c(V \setminus R)$ and $\mathcal{T}_{\mathcal{K}^c_1}$ is its clique tree.
\end{proof}

Next, we show that the edge set of biclique $\{L, R\}$ is not empty.

\begin{claim} \label{cm:sep_3}
Both $L$ and $R$ are not empty.
\end{claim}

\begin{proof}
Let the two ends of edge $e$ in $\mathcal{T}_{\mathcal{K}^c}$ to be $K'$ and $K''$. Since both $K'$ and $K''$ are maximal cliques of $G^c$ and $\MID(e) = K' \cap K''$, then both $K' \setminus \MID(e)$ and $K'' \setminus \MID(e)$ are not empty. We can complete the proof since $K' \in \mathcal{K}_1^c$ and $K'' \in \mathcal{K}_2^c$.
\end{proof}

Next, we show that the output of Algorithm~\ref{alg:bp_biclique_sep} is a biclique partition of a co-chordal graph $G$.

\begin{theorem} \label{thm:bp_correctness}
Given a co-chordal graph $G$ and clique tree $\mathcal{T}_{\mathcal{K}^c}$ of its complement $G^c$, the output of $\Call{FindPartition}{\mathcal{T}_{\mathcal{K}^c}}$ is a biclique partition of $G$.
\end{theorem}

\begin{proof}
We will use induction to prove Theorem~\ref{thm:bp_correctness}. In the base step, if $\mathcal{T}_{\mathcal{K}^c}$ only has one vertex, then $G^c$ is a complete graph and $G$ is an empty graph. Thus, $\bp(G) = 0$.

In the induction step, supposed that $\Call{FindPartition}{\mathcal{T}_{\mathcal{K}^c}}$ is a biclique partition of $G$ if $|V(\mathcal{T}_{\mathcal{K}^c})| < k$. Then, we consider the scenario that $\mathcal{T}_{\mathcal{K}^c}$ has $k$ vertices. By Proposition~\ref{prop:bp_separation}, an arbitrary edge $e$ cuts $\mathcal{T}_{\mathcal{K}^c}$ into two components $\mathcal{T}_{\mathcal{K}^c_1}$ and $\mathcal{T}_{\mathcal{K}^c_2}$, where we can construct a partitioning biclique $\{L, R\}$ as in Algorithm~\ref{alg:bp_biclique_sep} that partitions the edges of $G$ into the edge sets of $\{L, R\}$, $G(V \setminus L)$, and $G(V \setminus R)$. Moreover, $\mathcal{T}_{\mathcal{K}^c_1}$ and $\mathcal{T}_{\mathcal{K}^c_2}$ are clique trees of $G^c(V \setminus R)$, and $G^c(V \setminus L)$ respectively. Then, $\Call{FindPartition}{\mathcal{T}_{\mathcal{K}^c_1}}$ returns a biclique partition of $G(V \setminus R)$ and $\Call{FindPartition}{\mathcal{T}_{\mathcal{K}^c_2}}$ returns a biclique partition of $G(V \setminus L)$. Therefore, $\{\{L, R\}\} \cup \Call{FindPartition}{\mathcal{T}_{\mathcal{K}^c_1}} \cup \Call{FindPartition}{\mathcal{T}_{\mathcal{K}^c_2}}$ is a biclique partition of $G$.
\end{proof}

Finally, we prove that the size of the output of Algorithm~\ref{alg:bp_biclique_sep} is equal to one less than the number of maximal cliques of $G^c$.

\begin{theorem} \label{thm:bp_number_cochordal}
Given a co-chordal graph $G$ (with at least one vertex) and clique tree $\mathcal{T}_{\mathcal{K}^c}$ of its complement $G^c$, the output of $\Call{FindPartition}{\mathcal{T}_{\mathcal{K}^c}}$ is a biclique partition of $G$ with size $\mc(G^c) - 1$.
\end{theorem}

\begin{proof}
By Theorem~\ref{thm:bp_correctness}, we know that the output of $\Call{FindPartition}{\mathcal{T}_{\mathcal{K}^c}}$ is a biclique partition of $G$. We then want to prove that the size of the output of $\Call{FindPartition}{\mathcal{T}_{\mathcal{K}^c}}$ is $\mc(G^c) - 1$ by induction.

In the base step, $|V(\mathcal{T}_{\mathcal{K}^c})| = 1$, then $\mc(G^c) = 1$ since $G^c$ is a complete graph. The output is an empty set, which has size of 0.

In the induction step, assume that the size of the output of $\Call{FindPartition}{\mathcal{T}_{\mathcal{K}^c}}$ is $|V(\mathcal{T}_{\mathcal{K}^c})| - 1$ if $|V(\mathcal{T}_{\mathcal{K}^c})| \in \{1, 2, \hdots, k-1\}$. Then, if $|V(\mathcal{T}_{\mathcal{K}^c})| = k$, the output is $\{\{L, R\}\} \cup \Call{FindPartition}{\mathcal{T}_{\mathcal{K}^c_1}} \cup \Call{FindPartition}{\mathcal{T}_{\mathcal{K}^c_2}}$ where $L, R, \mathcal{T}_{\mathcal{K}^c_1}$ and $\mathcal{T}_{\mathcal{K}^c_2}$ are defined in Algorithm~\ref{alg:bp_biclique_sep}. Since $|V(\mathcal{T}_{\mathcal{K}^c_1})| + |V(\mathcal{T}_{\mathcal{K}^c_2})| = k$ and $|V(\mathcal{T}_{\mathcal{K}^c_1})|, |V(\mathcal{T}_{\mathcal{K}^c_2})| > 0$, we know that $|V(\mathcal{T}_{\mathcal{K}^c_1})| < k$ and $|V(\mathcal{T}_{\mathcal{K}^c_2})| < k$. Also both $L$ and $R$ are not empty by Claim~\ref{cm:sep_3}, the size of the output of $\Call{FindPartition}{\mathcal{T}_{\mathcal{K}^c}}$ is $1 + |V(\mathcal{T}_{\mathcal{K}^c_1})| - 1 + |V(\mathcal{T}_{\mathcal{K}^c_2})| - 1 = k - 1$.
\end{proof}

Theorem~\ref{thm:bp_number_cochordal} leads us to a direct result that the biclique partition number of a co-chordal graph $G$ is less than the number of maximal cliques of $G^c$.

\begin{corollary} \label{cor:bp_upp_co_chordal}
If $G$ is a co-chordal graph, then $\bp(G) \leq \mc(G^c)-1$.
\end{corollary}

The total time complexity of Algorithm~\ref{alg:bp_biclique_sep} is $O[|V|^3]$, which includes both of taking the complement of a co-chordal graph $G$ and computing a clique tree of $G^c$.

\begin{theorem} \label{thm:alg_bc_sep_time}
Algorithm~\ref{alg:bp_biclique_sep} runs in $O[|V|^3]$.
\end{theorem}
\begin{proof}
Taking the complement of a graph requires $O(|V|^2)$ time. A clique tree of $G^c$ can be computed in $O(|V|^3)$ using an algorithm described in Figure 4.2~\cite{blair1993introduction}.

Since the number of maximal cliques in the chordal graph $G^c$ is at most $|V|$, the number of edges in $\mathcal{T}_{\mathcal{K}^c}$ is $O(|V|)$. Each $e$ in $\mathcal{T}_{\mathcal{K}^c}$ is selected once so that the resulting subtrees are trees with one vertex. Thus, the total number of \Call{FindPartition}{$\mathcal{T}_{\mathcal{K}^c}$} called is $O(|V|)$. Since the size of a maximal clique in $G^c$ is at most $|V|$, it takes $O(|V|^2)$ to compute $L$ and $R$ in \Call{FindPartition}{$\mathcal{T}_{\mathcal{K}^c}$}.
\end{proof}







\section{A Heuristic Based on Finding Moplexes} \label{sec:h_peo}

Although Algorithm~\ref{alg:bp_biclique_sep} can give a biclique partition of a co-chordal graph $G$, it requires a clique tree of its complementary chordal graph $G^c$ and computes unions of some maximal cliques, which could include some redundant computations. Instead of using clique tree, we can use a moplex in $G^c$ to find a biclique partition. Figure~\ref{fig:clique_tree} motivates the idea via a simple example. In Algorithm~\ref{alg:bp_biclique_sep}, we can select an arbitrary edge in the clique tree to find a biclique. Thus, we can always select an edge that is incident to a leaf node of the tree. For example, if we select the edge with middle set $\{x_3\}$, the biclique built in the current call of \Call{FindPartition}{$\mathcal{T}$} has $L = \{x_4,x_5,x_6,x_7\}$ and $R=\{x_1, x_2\}$, where $\{x_1, x_2\}$ is a moplex of $G^c$. We can also see that selecting the other two edges can also lead to bicliques with a moplex as $L$ or $R$. It motivates us to design another heuristic to find biclique partitions on co-chordal graphs by finding moplexes, which can be found efficiently by LexBFS. In this section, we show that Algorithm~\ref{alg:bp_moplex_lexbfs} can return a biclique partition of a co-chordal graph $G$ with a same size as Algorithm~\ref{alg:bp_biclique_sep} in $O[|V|(|V|+|E^c|)]$ time.

\begin{figure}[H]
    \centering
\tikzset{every picture/.style={line width=0.75pt}} 

\begin{tikzpicture}[x=0.75pt,y=0.75pt,yscale=-0.8,xscale=0.8]

\draw   (13,198.65) -- (109.13,198.65) -- (109.13,232.5) -- (13,232.5) -- cycle ;
\draw   (181.23,198.65) -- (277.37,198.65) -- (277.37,232.5) -- (181.23,232.5) -- cycle ;
\draw   (347.87,198.65) -- (444,198.65) -- (444,232.5) -- (347.87,232.5) -- cycle ;
\draw   (181.23,107) -- (277.37,107) -- (277.37,140.85) -- (181.23,140.85) -- cycle ;
\draw    (229.3,139.81) -- (229.3,199.17) ;
\draw    (109.13,214.79) -- (182.04,215.84) ;
\draw    (277.37,214.79) -- (347.07,215.84) ;

\draw (18.07,203.32) node [anchor=north west][inner sep=0.75pt]    {$\{x_{1} ,x_{2} ,x_{3}\} \ $};
\draw (186.3,203.32) node [anchor=north west][inner sep=0.75pt]    {$\{x_{3} ,x_{4} ,x_{5}\} \ $};
\draw (353.73,203.32) node [anchor=north west][inner sep=0.75pt]    {$\{x_{3} ,x_{5} ,x_{6}\} \ $};
\draw (199,111.66) node [anchor=north west][inner sep=0.75pt]    {$\{x_{4} ,x_{7}\} \ $};
\draw (227.83,154.37) node [anchor=north west][inner sep=0.75pt]    {$\{x_{4}\} \ $};
\draw (125.28,187.69) node [anchor=north west][inner sep=0.75pt]    {$\{x_{3}\} \ $};
\draw (283.12,184.57) node [anchor=north west][inner sep=0.75pt]    {$\{x_{3} ,x_{5}\} \ $};

\end{tikzpicture}
    \caption{A clique tree $\mathcal{T}$ of some chordal graph $G^c$ with maximal cliques on the vertices and middle sets on the edges.}
    \label{fig:clique_tree}
\end{figure}
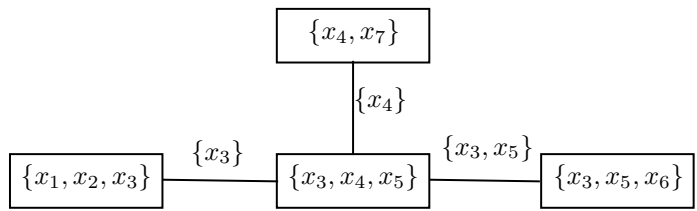

\begin{algorithm}[H]
\begin{algorithmic}[1]
\State \textbf{Input}: Co-chordal graph $G = (V, E)$.
\State \textbf{Output}: A biclique partition $\bp$ of $G$. 
\State Initialize $\bp \leftarrow \{\}$.
\State $G' \leftarrow G^c$ so it is a chordal graph.
\While{$|V(G')| \geq 1$} \label{ln:bp_peo_while_start}
\State Let $\sigma$ be the perfect elimination ordering of $G'$ obtained by LexBFS (Algorithm~\ref{alg:lexbfs}).
\State $L \leftarrow V(G') \setminus N_{G'}[\sigma(1)]$. 
\State $R \leftarrow \{u \in N_{G'}[\sigma(1)]: N_{G'}(u) \cap L = \emptyset\}$. \Comment{$R$ is a moplex containing $\sigma(1)$.} \label{ln:bp_peo_R}
\State If $L$ is not empty, $\bp \leftarrow \bp \cup \{\{L, R\}\}$.
\State Remove vertices in $R$ from $G'$. \label{ln:peo_op}
\EndWhile \label{ln:bp_peo_while_end}
\State \textbf{return} $\bp$
\end{algorithmic}
\caption{A moplex heuristic to find a biclique partition on co-chordal graph $G$.} \label{alg:bp_moplex_lexbfs}
\end{algorithm}

Note that Xu et al.\cite{xu2013moplex} also apply lexicographical depth-first search (LexDFS) to find moplexes, which can be an alternative method to find moplexes than LexBFS in Algorithm~\ref{alg:bp_moplex_lexbfs}.

Before we can prove that Algorithm~\ref{alg:bp_moplex_lexbfs} returns a biclique partition of a co-chordal graph, we need to show that $\{L, R\}$ is a partitioning biclique of $G'^c$ with vertices $V'$. In addition, one of the two induced subgraphs, $G'^c(V' \setminus L)$, has an empty edge set.

\begin{proposition} \label{prop:peo_correct}
Given a chordal graph $G' = (V', E')$ with at least two maximal cliques and $\sigma$ obtained by LexBFS, let $L = V' \setminus N_{G'}[\sigma(1)]$ and $R = \{u \in N_{G'}[\sigma(1)]: N_{G'}(u) \cap L = \emptyset\}$. Then, the edges of $G'^c$, the complement of $G'$, can be partitioned into edges in a biclique $E(\{L, R\})$ and edges in a chordal graph $E[G'^c(V' \setminus R)]$.
\end{proposition}

\begin{proof}
Since $\sigma$ is obtained by LexBFS (Algorithm~\ref{alg:lexbfs}), $\sigma$ is a perfect elimination ordering. Then, the closed neighborhood $N_{G'}[\sigma(1)]$ is a clique of $G'$. Thus, $V' \setminus (L \cup R)$ is an independent set of $G'^c$. By Proposition~\ref{prop:partitioned_biclique}, $\{L, R\}$ is a partitioning biclique of $G'^c$. Thus, the edges of $G'^c$ is partitioned into $E(\{L, R\})$, $E[G'^c(V' \setminus L)]$, and $E[G'^c(V' \setminus R)]$.
Since $V' \setminus L = N_{G'}[\sigma(1)]$ is a clique of $G'$, $E[G'^c(V' \setminus L)] = \emptyset$. Hence, the edges of $G'^c$ is partitioned into $E(\{L, R\})$ and $E[G'^c(V' \setminus R)]$.
\end{proof}

Next, we show that the output of Algorithm~\ref{alg:bp_moplex_lexbfs} is a biclique partition of a co-chordal graph $G$ with a size of $\mc(G^c)-1$.

\begin{proposition} \label{prop:peo_mc_count}
Given a chordal graph $G' = (V', E')$ with at least two maximal cliques and $\sigma$ obtained by LexBFS, let $L = V' \setminus N_{G'}[\sigma(1)]$, $R = \{u \in N_{G'}[\sigma(1)]: N_{G'}(u) \cap L = \emptyset\}$ and $G'_R = G'(V' \setminus R)$. Then, $R$ is a moplex of $G'$ and $\mc(G') = \mc(G'_R) + 1$.
\end{proposition}

\begin{proof}

The conclusion that $R$ is a moplex of $G$ can be drawn from Proposition~\ref{prop:sigma_1_moplex} (Theorem 5.1~\cite{berry1998separability}). The reader is referred to~\cite{berry1998separability} for the detailed proof. We want to note that $R$ is both a module and a clique. 

Since $\sigma$ is obtained by LexBFS in Algorithm~\ref{alg:lexbfs}, $\sigma$ is a perfect elimination ordering. Then, $(V' \setminus L)$ is a maximal clique of $G'$. Since $\sigma(1) \not\in V(G'_R)$, $G'(V' \setminus L)$ is not a subgraph of $G'_R$.

Given an arbitrary maximal clique $K$ of $G'$ that is not $(V' \setminus L)$, there exists $v'$ in clique $K$ such that $v'$ is not $\sigma(1)$ or its neighbor. Thus, $u \not\in R$ for every $u \in K$. Hence, $K$ is also a maximal clique subgraph of $G'_R$. Thus, $\mc(G'_R) \geq \mc(G') - 1$. 

Since $R$ is a moplex, $S = N_{G'}[\sigma(1)] \setminus R$ is a minimal separator of $G'$. Note that $(V' \setminus L) = N_{G'}[\sigma(1)]$ is a maximal clique of $G'$. Since $G'$ is chordal, there must exists a maximal clique $K$ of $G'$ that is not $(V' \setminus L)$ such that $S \subseteq K$ (See Theorem 4.1~\cite{blair1993introduction}). Note that $K$ is a maximal clique of $G'_R$. Therefore, $S$ is not a maximal clique of $G'_R$ and every maximal clique in $G'_R$ is also a maximal clique of $G'$. Hence, $\mc(G') = \mc(G'_R) + 1$. 
\end{proof}

\begin{theorem}
Algorithm~\ref{alg:bp_moplex_lexbfs} returns a biclique partition of $G$ with a size of $\mc(G^c)-1$ given $G$ is a co-chordal graph. 
\end{theorem}

\begin{proof}
By Proposition~\ref{prop:peo_correct}, the output of Algorithm~\ref{alg:bp_moplex_lexbfs} is a biclique partition of $G$. By Proposition~\ref{prop:peo_mc_count}, the number of bicliques in the biclique partition is $\mc(G^c)-1$. We want to note that $L=\emptyset$ if $\mc(G^c) = 1$, so the number is $\mc(G^c)-1$ instead of $\mc(G^c)$.
\end{proof}

We also show that Algorithm~\ref{alg:bp_moplex_lexbfs} is efficient: $O[|V|(|V|+|E^c|)]$ time.

\begin{theorem}
Algorithm~\ref{alg:bp_moplex_lexbfs} runs in $O[|V|(|V|+|E^c|)]$ time.
\end{theorem}

\begin{proof}
Taking the complement of a graph requires $O(|V|^2)$ time. The outer while loop has at most $O(|V|)$ iterations. Finding a perfect elimination ordering $\sigma$ of a chordal graph $G^c$ can be done in $O(|V| + |E^c|)$ time by LexBFS with partition refinement. The construction of $L$ takes $O(|V|)$ time. Computing $R$ in Line~\ref{ln:bp_peo_R} takes $O(|E^c|)$, where each edge in $G'$ can be traversed in constant times in the worst case.
\end{proof}



\section{Lower Bounds of Biclique Partition Numbers on Some Subclasses of Co-Chordal Graphs} \label{sec:lower}

In this section, we will first prove that both Algorithms~\ref{alg:bp_biclique_sep} and~\ref{alg:bp_moplex_lexbfs} find a minimum biclique partition of $G$ if its complement $G^c$ is both chordal and clique vertex irreducible. We start with the well-known Graham-Pollak Theorem, which provides a lower bound on the biclique partition number for complete graphs.

\begin{theorem} [Graham–Pollak Theorem~\cite{graham1971addressing, graham1972embedding}]
The edge set of the complete graph, $K_n$, cannot be partitioned into fewer than $n-1$ biclique subgraphs.
\end{theorem}



We first show that the biclique partition number of a graph is no less than the biclique partition number of its induced subgraph.

\begin{lemma} \label{lm:bp_subgraph}
For every induced subgraph $G'$ of a graph $G$, $\bp(G') \leq \bp(G)$.
\end{lemma}

\begin{proof}
Suppose that $\{\{L^i, R^i\}\}_{i=1}^{\bp(G)}$ is a minimum biclique partition of graph $G$. Then, it is trivial to show that $\{\{L^i \cap V', R^i \cap V'\}\}_{i=1}^{\bp(G)}$ is a biclique partition of graph $G'$. Thus, $\bp(G') \leq \bp(G)$.
\end{proof}




Then, we can show that the biclique partition number of an arbitrary graph is no less than its clique number minus one.
\begin{proposition} \label{prop:bp_omega}
Given a graph $G=(V, E)$, $\bp(G) \geq \omega(G) - 1$.
\end{proposition}

\begin{proof}
Let $K$ be a maximum clique of $G$ with size $\omega(G)$ and $G_K$ be the corresponding induced subgraph. By Graham–Pollak Theorem, we know that $\bp(G_K) \geq \omega(G) - 1$. By Lemma~\ref{lm:bp_subgraph}, $\bp(G) \geq \bp(G_K) \geq \omega(G) - 1$.
\end{proof}

Next, we prove that Algorithms~\ref{alg:bp_biclique_sep} and~\ref{alg:bp_moplex_lexbfs} are exact when the input graph $G$ has a chordal and clique vertex irreducible complement.

\begin{lemma} \label{lm:cvi}
If a graph $G$ is clique vertex irreducible, then $\mc(G) = \omega(G^c)$.
\end{lemma}

\begin{proof}
Since $G$ is clique vertex irreducible, there exists a vertex $v_i$ in maximal clique $K_i$ such that $K_i$ is the only maximal clique in $G$ contains $v_i$. Then, $\{v_i\}_{i=1}^{\mc(G)}$ is an independent set in $G$ so it is a clique in $G^c$. Thus, $\mc(G) \leq \omega(G^c)$. Furthermore, assume that $\omega(G^c) > \mc(G)$ and let $K^c$ be a maximum clique of $G^c$. Then, there exist two distinct vertices $u, v$ in $K^c$ such that $u$ and $v$ are in the same maximal clique of $G$. It implies $uv \in E(G)$, which is a contradiction.
\end{proof}

\begin{theorem}
Given a graph $G$ where its complement $G^c$ is chordal and clique vertex irreducible, $\bp(G) = \mc(G^c)-1$.
\end{theorem}

\begin{proof}
By Corollary~\ref{cor:bp_upp_co_chordal}, $\bp(G) \leq \mc(G^c) - 1$ since $G$ is a co-chordal graph. Since $G^c$ is clique vertex irreducible, then $\mc(G^c) = \omega(G)$. By Proposition~\ref{prop:bp_omega}, $\bp(G) \geq \mc(G^c) - 1$.
\end{proof}

Then, we derive a lower bound of the $\bp$ number on any split graph and prove that $\mc(G^c)-2 \leq \bp(G) \leq \mc(G^c)-1$ if $G$ is a split graph.

\begin{lemma} \label{lm:split_mc}
Given a split graph $G = (V, E)$, $V$ can be partitioned into two sets $V_1, V_2$ such that $V_1$ is a clique and $V_2$ is an independent set. Then, $|V_2| \leq \mc(G) \leq |V_2| + 1$.
\end{lemma}

\begin{proof}
Since $V_2$ is an independent set, any maximal clique of $G$ can include at most one vertex in $V_2$. Also, each vertex in $V_2$ is in at least one maximal clique. Then, $|V_2| \leq \mc(G)$.

We then want to prove that given an arbitrary vertex $v \in V_2$, there is at most one maximal clique of $G$ containing $v$. Assume that there exists two distinct maximal cliques of $G$, $K_1$ and $K_2$, containing $v$. Then, it is safe to conclude that both $K_1 \setminus \{v\} \subseteq V_1$ and $K_2 \setminus \{v\} \subseteq V_1$. Since $V_1$ is a clique in $G$, then we can construct a larger clique subgraph of $G$ by including all the vertices in $K_1$ and $K_2$, which is a contradiction. 

There is at most one maximal clique of $G$ not including any vertex in $V_2$, which is $V_1$. Therefore, $\mc(G) \leq |V_2| + 1$.
\end{proof}

Note that if $V_1$ is a maximal clique, then $\mc(G) = |V_2| + 1$. Then, we use Lemma~\ref{lm:split_mc} to prove $\mc(G^c) - 1 \leq \omega(G)$ and eventually $\bp(G) \geq \mc(G^c) - 2$ for an arbitrary split graph $G$.

\begin{theorem} \label{thm:bp_mc_minus_2}
Given a split graph $G$, then $\bp(G) \geq \mc(G^c) - 2$.
\end{theorem}

\begin{proof}
We first claim that the clique number of $G$, $\omega(G)$, is no less than the number of maximal cliques, $\mc(G^c) - 1$, i.e. $\omega(G) \geq \mc(G^c) - 1$. Then, by Proposition~\ref{prop:bp_omega}, $\bp(G) \geq \omega(G) - 1 \geq \mc(G^c) - 2$.

To prove the claim, we first start with the definition of split graphs. Since $G = (V, E)$ is a split graph, then $V$ can be partitioned into two sets $V_1, V_2$ such that $V_1$ is a clique and $V_2$ is an independent set of $G$. Then, we know that $|V_1| \leq \omega(G)$. 

Furthermore, $G^c$ is also a split graph, $V_1$ is an independent set of $G^c$, and $V_2$ is a clique in $G^c$. By Lemma~\ref{lm:split_mc}, we know that $\mc(G^c) \leq |V_1| + 1$. Therefore, $\mc(G^c) - 1 \leq |V_1| \leq \omega(G)$.
\end{proof}

In the proof of Theorem~\ref{thm:bp_mc_minus_2}, we show that $\omega(G) \geq \mc(G^c) - 1$ if $G$ is a split graph.
In Figure~\ref{fig:split}, we show that it is possible that $\omega(G) = \mc(G^c) - 1$. In this case, we only have $\mc(G^c) - 2 \leq \bp(G) \leq \mc(G^c) - 1$. However, we can have the exact value of $\bp(G)$, $\bp(G) = \mc(G^c) - 1$, if $\omega(G) = \mc(G^c)$, or, equivalently, the vertices of $G$ can be partitioned into vertex sets $V_1$ and $V_2$ such that $V_1$ is a maximal (maximum) clique and $V_2$ is an independent set but not maximal.

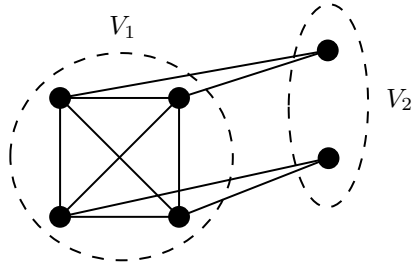
\begin{figure}[H]
    \centering
    \tikzset{every picture/.style={line width=0.75pt}} 

\begin{tikzpicture}[x=0.75pt,y=0.75pt,yscale=-1,xscale=1]

\draw  [dash pattern={on 4.5pt off 4.5pt}] (51.24,165.12) .. controls (51.54,136.24) and (76.9,113.09) .. (107.89,113.41) .. controls (138.88,113.74) and (163.75,137.41) .. (163.45,166.3) .. controls (163.15,195.18) and (137.78,218.33) .. (106.8,218.01) .. controls (75.81,217.68) and (50.93,194.01) .. (51.24,165.12) -- cycle ;
\draw  [fill={rgb, 255:red, 0; green, 0; blue, 0 }  ,fill opacity=1 ] (71.33,136.08) .. controls (71.33,133.32) and (73.57,131.08) .. (76.33,131.08) .. controls (79.09,131.08) and (81.33,133.32) .. (81.33,136.08) .. controls (81.33,138.84) and (79.09,141.08) .. (76.33,141.08) .. controls (73.57,141.08) and (71.33,138.84) .. (71.33,136.08) -- cycle ;
\draw  [fill={rgb, 255:red, 0; green, 0; blue, 0 }  ,fill opacity=1 ] (131.33,136.08) .. controls (131.33,133.32) and (133.57,131.08) .. (136.33,131.08) .. controls (139.09,131.08) and (141.33,133.32) .. (141.33,136.08) .. controls (141.33,138.84) and (139.09,141.08) .. (136.33,141.08) .. controls (133.57,141.08) and (131.33,138.84) .. (131.33,136.08) -- cycle ;
\draw  [fill={rgb, 255:red, 0; green, 0; blue, 0 }  ,fill opacity=1 ] (71.33,196.08) .. controls (71.33,193.32) and (73.57,191.08) .. (76.33,191.08) .. controls (79.09,191.08) and (81.33,193.32) .. (81.33,196.08) .. controls (81.33,198.84) and (79.09,201.08) .. (76.33,201.08) .. controls (73.57,201.08) and (71.33,198.84) .. (71.33,196.08) -- cycle ;
\draw  [fill={rgb, 255:red, 0; green, 0; blue, 0 }  ,fill opacity=1 ] (206.33,112.08) .. controls (206.33,109.32) and (208.57,107.08) .. (211.33,107.08) .. controls (214.09,107.08) and (216.33,109.32) .. (216.33,112.08) .. controls (216.33,114.84) and (214.09,117.08) .. (211.33,117.08) .. controls (208.57,117.08) and (206.33,114.84) .. (206.33,112.08) -- cycle ;
\draw  [fill={rgb, 255:red, 0; green, 0; blue, 0 }  ,fill opacity=1 ] (131.33,196.08) .. controls (131.33,193.32) and (133.57,191.08) .. (136.33,191.08) .. controls (139.09,191.08) and (141.33,193.32) .. (141.33,196.08) .. controls (141.33,198.84) and (139.09,201.08) .. (136.33,201.08) .. controls (133.57,201.08) and (131.33,198.84) .. (131.33,196.08) -- cycle ;
\draw    (76.33,136.08) -- (136.33,196.08) ;
\draw    (76.33,196.08) -- (136.33,196.08) ;
\draw    (76.33,136.08) -- (76.33,196.08) ;
\draw    (76.33,136.08) -- (136.33,136.08) ;
\draw    (136.33,136.08) -- (136.33,196.08) ;
\draw    (136.33,136.08) -- (76.33,196.08) ;
\draw  [fill={rgb, 255:red, 0; green, 0; blue, 0 }  ,fill opacity=1 ] (206.73,166.48) .. controls (206.73,163.72) and (208.97,161.48) .. (211.73,161.48) .. controls (214.49,161.48) and (216.73,163.72) .. (216.73,166.48) .. controls (216.73,169.24) and (214.49,171.48) .. (211.73,171.48) .. controls (208.97,171.48) and (206.73,169.24) .. (206.73,166.48) -- cycle ;
\draw  [dash pattern={on 4.5pt off 4.5pt}] (191.63,139.8) .. controls (191.92,111.51) and (201.12,88.67) .. (212.18,88.79) .. controls (223.24,88.91) and (231.96,111.93) .. (231.66,140.22) .. controls (231.37,168.51) and (222.16,191.35) .. (211.11,191.23) .. controls (200.05,191.11) and (191.33,168.09) .. (191.63,139.8) -- cycle ;
\draw    (211.73,166.48) -- (76.33,196.08) ;
\draw    (76.33,136.08) -- (211.33,112.08) ;
\draw    (136.33,136.08) -- (211.33,112.08) ;
\draw    (211.73,166.48) -- (136.33,196.08) ;

\draw (115.37,100.5) node   [align=left] {\begin{minipage}[lt]{21.35pt}\setlength\topsep{0pt}
$\displaystyle V_{1}$
\end{minipage}};
\draw (254.97,136.9) node   [align=left] {\begin{minipage}[lt]{21.35pt}\setlength\topsep{0pt}
$\displaystyle V_{2}$
\end{minipage}};

\end{tikzpicture}
    \caption{A split graph $G$ with vertex sets $V_1$ (a maximum clique) and $V_2$ (a maximum independent set), where $\omega(G) = 4$ but $\mc(G^c) = 5$.}
    \label{fig:split}
\end{figure}
\begin{remark}
If the vertices of a split graph $G$ can be partitioned into vertex sets $V_1$ and $V_2$ such that $V_1$ is a maximal (maximum) clique and $V_2$ is an independent set but not maximal, then $\bp(G) = \mc(G^c) - 1$.
\end{remark}

\section{Final remarks} \label{sec:fr}

If a graph $G = (V, E)$ is a co-chordal graph, the biclique partition number of $G$ is less than the number of maximal cliques of its complement $G^c$. Additionally, we provided two heuristics, one based on clique trees and one based on finding moplexes, to find an explicit construction of a biclique partition with a size of $\mc(G^c)-1$. We also showed that the computational time of the moplex heuristic is $O[|V|(|V|+|E^c|)]$.

If a graph $G$ where its complement $G^c$ is both chordal and clique vertex irreducible, then $\bp(G) = \mc(G^c)-1$. If a graph $G$ is a split graph, another subclass of co-chordal, then we have $\mc(G^c)-2 \leq bp(G) \leq \mc(G^c)-1$.


In Section~\ref{sec:h_ct}, we showed that given a co-chordal graph $G$ and a clique tree of $G^c$, $\mathcal{T}_{\mathcal{K}^c}$, Algorithm~\ref{alg:bp_biclique_sep} can return a biclique partition with a size of $\mc(G^c)-1$ no matter which edge of $\mathcal{T}_{\mathcal{K}^c}$ is selected in each recursion. An open question is whether $\bp(G) = \mc(G^c)-1$ if $G$ is a co-chordal graph or split graph. If it is the case, it is an extension of Graham-Pollak theorem to a more general class of graphs.

\section*{Acknowledgements}

The authors would like to thank the reviewers and Bo Jones for their helpful and insightful comments.

\newpage
\bibliography{mybibfile}




\end{document}